\documentclass[11pt]{article}
\usepackage{amsmath, amssymb, amsthm}
\usepackage{bbm}
\usepackage{enumitem}
\usepackage{hyperref}
\usepackage{tikz}
\usetikzlibrary{calc}
\usepackage{xcolor}
\usepackage{multirow}
\usepackage{comment}
\usepackage{authblk}

\newtheorem{theorem}{Theorem}[section]
\newtheorem{lemma}{Lemma}[section]

\theoremstyle{remark}

\newcommand{\N}{\mathbb{N}}
\newcommand{\R}{\mathbb{R}}

\newcommand{\Z}{\mathbb{Z}}

\usepackage[left=1in, right=1in, bottom = 1in, top = 1in]{geometry}
\title{Law of the  Iterated Logarithm for $p$-Walks on $\mathbb{Z}$}
\author{Robin Kaiser}

\date{}
\begin{document}

\maketitle
\begin{abstract}
    The $p$-rotor walk on $\mathbb{Z}$ is a self-interacting walk that interpolates between the simple random walk and the deterministic rotor walk. While the weak convergence of this model to a perturbed Brownian motion is known, its almost sure asymptotic boundaries have not been characterized. In this paper, we establish the exact Law of the Iterated Logarithm (LIL) for the $p$-rotor walk. Utilizing the decomposition of the walk into a martingale perturbed by its running extrema, we obtain first a functional Law of the Iterated Logarithm for the linearly interpolated paths of the $p$-walk. We then obtain the classical LIL constants by solving a calculus of variations problem over the perturbed Strassen set.
\end{abstract}

{\it Keywords:} $p$-walks, integers, law of the iterated logarithm, variational calculus, Strassen, functional limits, perturbation.

{2020 Mathematics Subject Classification.} 60G42, 60F05, 60K35.

\section{Introduction}
Originally introduced as the Eulerian walker in \cite{PDDK96}, and later popularized in the mathematical community by Propp \cite{P03} as a deterministic counterpart to simple random walks, the rotor-router model has garnered widespread interest from mathematicians. Despite its purely deterministic evolution in time, the aggregate of particles in the rotor-router model approximates the behaviour of simple random walks in remarkable detail \cite{CS06,HP10,LP08,LP09}. Concerning the behaviour of a single moving particle in the rotor dynamics, the recurrence and transience properties have been studied extensively on different graphs \cite{AH12,FLP16,HSH20,HMSH15,KSH24,AH11}; a problem which is still open on the two-dimensional Euclidean lattice due to its notorious difficulty.

In \cite{HLSH18}, the authors introduced a model, which approximates the behaviour of simple random walks even more closely, by interpolating between the deterministic rotor-router dynamics and the pervasive stochastic fluctuations of random walks: $p$-walks.  In this model, each vertex is equipped with a rotor. When the walker visits a vertex, the rotor "breaks" with probability $p \in (0,1)$, leaving its direction unchanged, or "flips" with probability $1-p$. The walker then steps in the direction of the updated rotor. When $p = 1/2$, the model perfectly recovers the simple random walk; as $p$ tends towards $0$, it mimics the absolute rigidity of the deterministic rotor walk.

The main result of \cite{HLSH18} was resolving the central limit behavior of this interpolating $p$-walk. By establishing a decomposition of the walk into a bounded martingale perturbed by its own historical extrema (up to a bounded error), they proved a functional central limit theorem (FCLT), showing that the scaled macroscopic path converges weakly to a perturbed Brownian motion.

However, while the distributional limits of the $p$-rotor walk on $\Z$ are now well understood, its almost sure pathwise fluctuations have remained an open question. Historically, the measure of a walk's extreme spatial fluctuations is the Law of the Iterated Logarithm (LIL). First formulated by Khintchine in 1924 \cite{K24} and Kolmogorov in 1929 \cite{K29}, and later elevated to a functional principle by Strassen in 1964 \cite{S64}, the LIL dictates the absolute asymptotic envelope that bounds the trajectories of the process almost surely. The purpose of this note is to complete the asymptotic picture of the $p$-rotor walk by establishing its exact Law of the Iterated Logarithm. 

For stating our main result, we introduce the $(a,b)$-random initial rotor configurations as in \cite{HLSH18}. We define for $a,b\in[0,1]$ the two-sided initial rotor configuration $(\rho_0(x))_{x\in\Z}$ such that
\begin{align*}
    \forall x>0:&\mathbb{P}(\rho_0(x)=1)=1-\mathbb{P}(\rho_0(x)=-1)=a,\\
    \forall x<0:&\mathbb{P}(\rho_0(x)=1)=1-\mathbb{P}(\rho_0(x)=-1)=1-b,\\
    &\mathbb{P}(\rho_0(0)=1)=1-\mathbb{P}(\rho_0(0)=-1)=\frac{1}{2}.
\end{align*}

\begin{theorem}\label{thm:main}
    Let $(X_n)_{n\in\N}$ be the $p$-rotor walk on $\mathbb{Z}$ with parameter $p \in (0,1)$ and $(a,b)$-random initial rotor configuration where $a, b \in [0,1]$. Then almost surely the following limits hold
    $$\limsup_{n \to \infty} \frac{X_n}{\sqrt{2n \log \log n}} = \frac{2p\sqrt{p(1-p)}}{p-(2p-1)a}\hspace{0.5cm}
    \text{and}\hspace{0.5cm}
    \liminf_{n \to \infty} \frac{X_n}{\sqrt{2n \log \log n}} = -\frac{2p\sqrt{p(1-p)}}{p-(2p-1)b}.$$
\end{theorem}

Because the model's underlying randomness is heavily influenced by the dynamic, self-interacting environment of the rotors, classical i.i.d. techniques are insufficient. Instead, we leverage the almost sure convergence of the rescaled predictable variation of the rotor mechanism in conjunction with the functional Law of the Iterated Logarithm for martingales, to reformulate the derivation of the constants in the Law of the Iterated Logarithm as a calculus of variations problem.

\textbf{Outline.} The remainder of this paper is structured as follows. In Section \ref{sec:prelim}, we formalize the martingale decomposition of $p$-walks and establish the necessary almost sure convergence for the predictable quadratic variation of the martingale part. We also introduce the perturbation mapping as well as the functional LIL our approach is based on. In Section \ref{sec:func}, we construct the continuous time interpolation and apply the functional LIL. Finally, in Section \ref{sec:var}, we evaluate the extremum of the transformed Strassen set to conclude the proof of Theorem \ref{thm:main}.

\section{Preliminaries}\label{sec:prelim}
\subsection{Results on $p$-Walks}
We define here $p$-walks on $\Z$ and collect some important results. Furthermore, we improve the convergence results of the quadratic variation from \cite{HLSH18} to almost sure convergence.

\textbf{$p$-Walks on $\Z$.} Let the state space of the walker be the integer lattice $\mathbb{Z}$. At any discrete time $n \in \mathbb{N}_0$, the environment is described by a rotor configuration $\rho_n \in \{-1, 1\}^{\mathbb{Z}}$, where $\rho_n(x)$ denotes the direction of the rotor at vertex $x$. Let $X_n \in \mathbb{Z}$ denote the position of the walker at time $n$, with the initial position fixed at $X_0 = 0$.

The dynamics of the $p$-rotor walk are governed by a sequence of independent Bernoulli random variables. Let $p \in (0,1)$ be fixed. When the walker visits a vertex $x = X_n$, the rotor at that vertex updates according to the rule:
$$\rho_{n+1}(x) = 
\begin{cases} 
\rho_n(x) & \text{with probability } p \\
-\rho_n(x) & \text{with probability } 1-p 
\end{cases}$$
For all other vertices $y \neq X_n$, the rotors remain unchanged: $\rho_{n+1}(y) = \rho_n(y)$. The walker then steps in the direction of the newly updated rotor:
$$X_{n+1} = X_n + \rho_{n+1}(X_n)$$
We define the natural filtration $\mathcal{F}_n = \sigma(X_0, \rho_0, \dots, X_n, \rho_n)$ containing the entire history of the walk and the dynamic environment up to time $n$.

\textbf{Martingale Decomposition of $p$-Walks.} A central result of \cite{HLSH18} demonstrates that the $p$-rotor walk path can be decomposed into a bounded-increment martingale perturbed by its historical extrema, when the initial rotors are given by $(a,b)$-random initial configurations for $a,b\in[0,1]$.

Let $\overline{X}_n = \max_{0 \leq k \leq n} X_k$ and $\underline{X}_n = \min_{0 \leq k \leq n} X_k$ denote the running supremum and infimum of the walk, respectively. We further define the difference increments as $\Delta_k=X_{k+1}-X_k$. We further decompose the difference sequences as
$$Y_n=\sum_{k=0}^{n-1}\Big(\Delta_k-\mathbb{E}[\Delta_k\;|\;\mathcal{F}_k]\Big),\hspace{1cm}Z_n=\sum_{k=0}^{n-1}\mathbb{E}[\Delta_k\;|\;\mathcal{F}_k].$$
Notice that $(Y_n)_{n\in\N}$ is a mean-zero martingale with respect to the filtration $(\mathcal{F}_n)_{n\in\N}.$ Furthermore, it holds that $X_n=Y_n+Z_n$. We collect here a central result we need for our proof. It can be found in \cite[Proposition 2.3.]{HLSH18}
\begin{lemma}\label{lem:martingale-decomp}
    For the $p$-walk $(X_n)_{n\in\N}$ with $(a,b)$-random initial configuration with $a,b\in[0,1]$ it holds that
    $$X_n=W_n+\alpha\overline{X}_n+\beta\underline{X}_n.$$
    Here, $\alpha$ and $\beta$ depend on the initial configuration and are given by
    $$\alpha=a\frac{2p-1}{p},\hspace{1cm}\beta=b\frac{2p-1}{p}.$$
    Furthermore, $W_n$ for $n\in\N$ is given by
    $$W_n=\frac{1}{2p}(Y_{n+1}-\Delta_n).$$
\end{lemma}
We want to emphasize here, that it holds for the constants $\alpha,\beta$ appearing in Lemma \ref{lem:martingale-decomp}, that they are strictly less then $1.$ In fact, because $p\in(0,1)$ it holds
$\alpha=a\frac{2p-1}{p}=a(2-\frac{1}{p})<a\leq 1,$
and analogously for $\beta.$

\textbf{Quadratic Variation.} Another important property that we need is, that the quadratic variation of the martingale $(Y_n)_{n\in\N}$ converges to a constant almost surely when divided by $n$. Let us define the quadratic variation process as
$$V_n=\sum_{k=1}^{n}\mathbb{E}\left[\left(\Delta_k-\mathbb{E}[\Delta_k\;|\;\mathcal{F}_k]\right)^2|\mathcal{F}_k\right].$$
In \cite{HLSH18}, it was only shown that convergence holds in probability, as this suffices to apply the function central limit theorem for martingales. We thus need to upgrade the convergence results, as it is crucial for the application of the function LIL that the quadratic variation converges almost surely.
\begin{lemma}\label{lem:martingale-convergence}
    Consider the martingale $(Y_n)_{n\in\N}$. It then holds that
    $$\lim_{n\rightarrow\infty}\frac{\max\{|Y_k|\;|\;0\leq k\leq n\}}{n}=0\hspace{1cm}\text{almost surely}.$$
\end{lemma}
\begin{proof}
    Notice that $(Y_n)_{n\in\N}$ is a mean-zero martingale with bounded differences. Thus, by the martingale law of large numbers it holds that
    $$\frac{|Y_n|}{n}\xrightarrow{n\rightarrow\infty}0,$$
    almost surely. The analogue statement for the maximum now follows from standard analysis techniques. Since the sequence $(|Y_n|)_{n\in\N}$ converges to $0$, we can find for any $\varepsilon>0$ some number $N_0\in\N$ such that for all $n\geq N_0$ it holds that $|Y_n|\leq \varepsilon n$. Furthermore, for the initial segment $n<N_0$ the maximum is simply come constant $C_{N_0}=\max\{|Y_k|\;\;0\leq k\leq N_0\}.$ We thus obtain
    $$\frac{\max\{|Y_k|\;|\;0\leq k\leq n\}}{n}\leq \varepsilon+\frac{C_{N_0}}{n},$$
    which converges to $\varepsilon$. Since $\varepsilon>0$ is arbitrary, we obtain the desired claim.
\end{proof}
Let us write for convenience $\max\{|W_k|\;|\;0\leq k\leq n\}=\overline{|W|}_n.$ We can now prove almost sure convergence for the quadratic variation, by comparing the running maximum of the martingale term to the running maximum and minimum of the $p$-walk.
\begin{lemma}\label{lem:quadratic-variation}
    Let the $p$-walk $(X_n)_{n\in\N}$ with $(a,b)$-random initial configuration with $a,b\in[0,1]$ be given, and denote by $V_n$ for $n\in\N$ the qudratic variation of the martingale $(Y_n)_{n\in\N}$. It then holds almost surely
    $$\lim_{n\rightarrow\infty}\frac{V_n}{n}=4p(1-p).$$
\end{lemma}
\begin{proof}
    We start by comparing the running maximum and minimum of the $p$-walk to the martingale part $(Y_n)_{n\in\N}$. Let for $n\in\N$ denote by $T_n$ the first time the current maximum was achieved by the $p$-walk, that is
    $$T_n:=\min\{k\leq n\;|\;\overline{X}_k=\overline{X}_n\}.$$
    If we now define $\beta^-=\max\{0,-\beta\}$, it then holds that
    $$X_{T_n}=W_{T_n}+\alpha X_{T_n}+\beta \underline{X}_{T_n}\Rightarrow (1-\alpha)X_{T_n}=W_{T_n}+\beta\underline{X}_{T_n}.$$
    From this we obtain
    $$(1-\alpha)\overline{X_n}\leq \overline{|W|}_n+\beta^-|\underline{X}_n|,$$
    where we used the triangle inequality, as well as the fact that the running minimum is decreasing. Analogously, we obtain for the running mininum the inequality
    $$(1-\beta)|\underline{X}_n|\leq \overline{|W|}_n+\alpha^-\overline{X}_n,$$
    where $\alpha^-=\max\{0,-\alpha\}$. If we now insert the inequality for the running minimum into the inequality for the running maximum, we obtain
    $$\overline{X}_n\leq \left(\frac{1}{1-\alpha}+\frac{\beta^-}{(1-\beta)(1-\alpha)}\right)\overline{|W|}_n+\frac{\alpha^-\beta^-}{(1-\beta)(1-\alpha)}\overline{X}_n.$$
    Notice now that if either $\alpha\geq 0$ or $\beta\geq 0$, we have that
    $$\frac{\alpha^-\beta^-}{(1-\beta)(1-\alpha)}=0.$$
    On the other hand, if both numbers are negative we have
    $$\frac{\alpha^-\beta^-}{(1-\beta)(1-\alpha)}=\frac{|\alpha||\beta|}{1+|\alpha|+|\beta|+|\beta||\alpha|}<1.$$
    We thus obtain
    $$\overline{X}_n\leq C_{\alpha,\beta}\overline{|W|}_n,\hspace{1cm}\text{where}\hspace{1cm}C_{\alpha,\beta}>0.$$
    Since $W_n$ is simply the martingale $Y_n$ plus a bounded increment term, we obtain therefore from Lemma \ref{lem:martingale-convergence} that
    $$\lim_{n\rightarrow\infty}\frac{\overline{X}_n}{n}=0\hspace{1cm}\text{almost surely}.$$
    Via the same argument, we obtain the analogous statement for the running minimum. The result now follows from the representation of the quadratic variation as derived in the proof of \cite[Theorem 2.7.]{HLSH18}
    $$V_n=n-(2p-1)^2\left(-(1-2b)\underline{X}_n+(1-2a)\overline{X}_n+\underline{X}_n-\overline{X}_n+1\right).$$
    If we divide this expression by $n$, and use the limit as derived for the running minimum and maximum, we obtain
    $$\lim_{n\rightarrow\infty}\frac{V_n}{n}=1-(2p-1)^2=4p-4p^2=4p(1-p)\hspace{1cm}\text{almost surely}.$$
\end{proof}

\subsection{Perturbation Mapping and Functional LIL}
In this section we define the perturbation mapping and introduce the functional Law of the Iterated Logarithm for martingales.

\textbf{Continuous Perturbation Map.} Let $C_{ac}([0,1])$ denote the space of absolutely continuous real-valued functions on $[0,1]$ equipped with the uniform topology induced by the supremum norm $||f||_\infty = \sup_{0 \leq t \leq 1} |f(t)|$.

We define for $\alpha,\beta\in\R$ the double perturbation mapping $\Phi_{\alpha,\beta}: C_{ac}([0,1]) \to C([0,1])$. For a given path $f \in C_{ac}([0,1])$ with $f(0)=0$, $g = \Phi_{\alpha,\beta}(f)$ is defined as the unique continuous solution to the pathwise equation
\begin{align}\label{eq:perturbation-mapping}\tag{PM}
g(t) = f(t) + \alpha \sup_{0 \leq s \leq t} g(s) + \beta \inf_{0 \leq s \leq t} g(s).
\end{align}
Provided $\alpha < 1$ and $\beta < 1$, the mapping $\Phi_{\alpha,\beta}$ is well-defined. We want to note here, that without the additional assumption of absolute continuity, there could exist more than one solution to the equation (\ref{eq:perturbation-mapping}), see \cite{D96}. Another important property that we need for the perturabtion mapping is, that it preserves absolute continuity. Let us introduce for $f\in C_{ac}([0,1])$ the notation for the running maximum and minimum
$$\overline{f}(t)=\max\{f(s)\;|\;s\in[0,t]\},\hspace{1cm}\underline{f}(t)=\min\{f(x)\;|\;s\in[0,t]\}.$$
\begin{lemma}\label{lem:absolute-cont-preserved}
    The perturbation mapping $\Phi_{\alpha,\beta}$ for $\alpha,\beta<1$ preserves absolute continuity.
\end{lemma}
\begin{proof}
    Let $f\in C_{ac}([0,1])$ be absolutely continuous. Furthermore, set $g=\Phi_{\alpha,\beta}(f).$ Notice that both $\overline{g}$ and $\underline{g}$ are continuous and monotone, and thus of bounded variation. Since this implies that $g$ can be written as a sum of functions of bounded variation, $g$ must also be of bounded variation. Since $g$ is also continuous, it remains to prove that the image of null sets under $g$ is again a null set. For this we will use that $f$ maps null sets to null sets.

    Let $Z\subseteq[0,1]$ be a set of measure $0$. Consider the sets $E^+=\{g=\overline{g}\}$ and $E^-=\{g=\underline{g}\}$. Firstly, it holds on $E^+\cap E^-$ that $g=0$, which is a single point and thus has measure $0$. Furthermore, notice that the complement of $E^-$ can be written as a disjoint union of countably many open intervals $(a_i,b_i)$. On each of these intervals, the running infimum is simply some constant $c_i$. Thus, we have on $E^+\cap(a_i,b_i)$ that $(1-\alpha)g=f+\beta c_i$. That is, $g$ is simply an affine transformation of $f$. We thus have
    $$\vert g(Z\cap E^+)\vert\leq\sum_{i\in\N}\vert(\frac{1}{1-\alpha}f+\beta c_i)(Z\cap (a_i,b_i)\cap E^+)\vert=0,$$
    where $\vert\cdot\vert$ denotes the Lebesgue measure of a set. Similarly, we obtain that $g(Z\cap E^-)$ is a null set. Finally, we notice that on each connected component of the complement of $E^+\cup E^-$, the running supremum and infimium are just a constant. We therefore also see that $g(Z\cap(E^+\cup E^-)^C)$ is a null set. Thus, we have
    $$\vert g(Z)\vert\leq \vert g(Z\cap E^+)\vert+\vert g(Z\cap E^-)\vert+\vert g(Z\cap (E^+\cup E^-)^C)\vert=0.$$
    Thus, $g$ is continuous, of bounded variation and maps null sets to null sets (it has the Luzin (N) property). It is thus absolutely continuous.
\end{proof}
Notice that this proof also implies that the running supremum and running infimum of $g$ are absolutely continuous.

\textbf{Functional Law of the Iterated Logarithm.} The fundamental tool for extracting the almost sure asymptotic envelope of the underlying martingale is the functional Law of the Iterated Logarithm. We define the Strassen set $\mathcal{K} \subset C([0,1])$ as the set of absolutely continuous functions starting at the origin with bounded Dirichlet energy:
$$\mathcal{K} = \left\{ f \in C([0,1]) : f(0)=0, \, \int_0^1 (f'(t))^2 dt \leq 1 \right\}.$$
The importance of the Strassen set is, that it appears as the set of accumulation points of the properly rescaled paths of the martingale term $(Y_n)_{n\in\N}$. Specifically, we need \cite[Theorem 4.7.]{HH80}, which we will state as the following lemma. It shows a functional Law of the Iterated Logarithm for martingale paths, which are linearly interpolated at suitable random times.
\begin{lemma}[Functional Law of the Iterated Logarithm for Martingales]\label{lem:martingale-LIL}
    In the following, we let $(\Sigma_n)_{n\in\N}$ be a filtration, $(S_n)_{n\in\N}$ be a martingale with respect to this filtration and we let $D_n=S_n-S_{n-1}$ be the martingale difference sequence. Furthermore, we let $(I_n)_{n\in\N}$ be a sequence of random variables with $0<I_1\leq I_2\leq\ldots$, and $(Z_n)_{n\in\N}$ be a sequence of non-negative random variables. Also, we assume that for all $n\in\N$ the random variables $Z_n$ and $I_n$ are $\Sigma_{n-1}$ measurable.

    We then define the path of the martingale interpolated at the random times $(I_n)_{n\in\N}$ as
    $$\mathcal{S}_n(t)=\frac{1}{\phi(I_n^2)}\left(S_i+\frac{tI_n^2-I_i^2}{I_{i+1}^2-I_i^2}D_{i+1}\right),\hspace{1cm}\text{for }I_i^2\leq tI_n^2<I_{i+1}^2\text{ and }i\leq n-1,$$
    where $\phi(t)=\sqrt{2t\log\log(t)}.$ Assume that the following statements hold:
    \begin{enumerate}
        \item The following limit holds almost surely 
        $$\lim_{n\rightarrow\infty}\frac{1}{\phi(I_n^2)}\sum_{i=1}^n\left(D_i\mathbbm{1}_{\{|D_i|>Z_i\}}-\mathbb{E}[D_i\mathbbm{1}_{\{|D_i|>Z_i\}}\;|\;\Sigma_{i-1}]\right)=0.$$
        \item The following limit holds almost surely
        $$\lim_{n\rightarrow\infty}\frac{1}{I_n^2}\sum_{i=1}^n\left(\mathbb{E}[D_i^2\mathbbm{1}_{\{|D_i|\leq Z_i\}}\;|\;\Sigma_{i-1}]-\left(\mathbb{E}[D_i\mathbbm{1}_{\{|D_i|\leq Z_i\}}\;|\;\Sigma_{i-1}]\right)^2\right)=1.$$
        \item The sum $\sum_{i=1}^n I_i^{-4}\mathbb{E}[D_i^4\mathbbm{1}_{\{|D_i|\leq Z_i\}}\;|\;\Sigma_{i-1}]$ is finite almost surely.
        \item For the interpolation times it holds almost surely that
        $$\frac{I_{n+1}}{I_n}\xrightarrow{n\rightarrow\infty}1,\hspace{1cm}I_n\xrightarrow{n\rightarrow\infty}\infty.$$
    \end{enumerate}
    Then almost surely it holds that, the sequence of interpolated paths $(\mathcal{S}_n)_{n\in\N}$ is relatively compact in $C([0,1])$ and its set of accumulation points coincides exactly with $\mathcal{K}.$
\end{lemma}
Originally, $\mathcal{K}$ was identified by Strassen in \cite{S64} as the set of accumulation points of properly rescaled paths of the Brownian motion. In his subsequent paper, \cite{S67} it was then proven that martingales, under sufficiently strong conditions, can be approximated closely by Brownian motions, from which the functional Law of the Iterated Logarithm for martingales follows. We decided to use the more general martingale LIL from \cite{HH80}, as it is better suited for our purposes. We want to remark here, that most of the fairly technical conditions in Lemma \ref{lem:martingale-LIL} reduce to simpler forms in the case of the $p$-walk, due to the almost sure convergence of the quadratic variation process and the bounded jump size.

\section{Functional Law of the Iterated Logarithm}\label{sec:func}
In this section we will prove a functional LIL for the properly rescaled paths of the $p$-walk. The prove relies on first using Lemma \ref{lem:martingale-LIL} on the martingale part of the decomposition in Lemma \ref{lem:martingale-decomp}, and then to carry over the result to the paths of the $p$-walk via the perturbation mapping as defined in (\ref{eq:perturbation-mapping}).

To analyze the asymptotic boundaries of the $p$-rotor walk, we first translate the discrete sequences $(X_n)_{n\in\N}$, $(Y_n)_{n\in\N}$ and $(W_n)_{n\in\N}$ into continuous-time paths. We define for $n\in\N$ the interpolation times $I_n=\sqrt{n}$ deterministically. We then define the piecewise-interpolated paths as we did in Lemma \ref{lem:martingale-LIL} as
\begin{align*}
    &\mathcal{X}_n(t)=\frac{1}{\sqrt{2n\log\log(n)}}\left(X_i+\frac{tn-i}{{i+1}-i}(X_{i+1}-X_i)\right),\hspace{1cm}\text{for }i\leq tn<{i+1}\text{ and }i\leq n-1,\\
    &\mathcal{Y}_n(t)=\frac{1}{\sqrt{2n\log\log(n)}}\left(Y_i+\frac{tn-i}{{i+1}-i}(Y_{i+1}-Y_i)\right),\hspace{1cm}\text{for }i\leq tn<{i+1}\text{ and }i\leq n-1,\\
    &\mathcal{W}_n(t)=\frac{1}{\sqrt{2n\log\log(n)}}\left(W_i+\frac{tn-i}{{i+1}-i}(W_{i+1}-W_i)\right),\hspace{1cm}\text{for }i\leq tn<{i+1}\text{ and }i\leq n-1.
\end{align*}
Notice that the linear interpolation factor can be shortened to
$\frac{tn-i}{i+1-i}=tn-i,$
however, we leave it in its extended form to properly highlight the connection to Lemma \ref{lem:martingale-LIL}.
\subsection{Functional LIL for the Underlying Martingale}
We consider first the path sequence $(\mathcal{Y}_n)_{n\in\N}$. Notice that the martingale $(Y_n)_{n\in\N}$ has increments bounded by some constant $C>0$ (specifically, we can choose $C=2$). Let us thus choose for all $n\in\N$ the random variable $Z_n=C$ deterministically. Then it trivially holds that $Z_n$ and $I_n$ are $\mathcal{F}_{n-1}$ measurable. We must now check the conditions of the martingale LIL in Lemma \ref{lem:martingale-LIL}.

\begin{enumerate}
    \item By our choice of the sequence $(Z_n)_{n\in\N}$ it holds that $\mathbbm{1}_{\{|Y_i-Y_{i-1}|>Z_i\}}=0$ for all $i\in\N$. Thus it follows trivially that almost surely
    $$\lim_{n\rightarrow\infty}\frac{1}{\sqrt{n\log\log(n)}}\sum_{i=1}^n\left((Y_i-Y_{i-1})\mathbbm{1}_{\{|Y_i-Y_{i-1}|>Z_i\}}-\mathbb{E}[(Y_i-Y_{i-1})\mathbbm{1}_{\{|Y_i-Y_{i-1}|>Z_i\}}\;|\;\mathcal{F}_{i-1}]\right)=0.$$

    \item Since $(Y_n)_{n\in\N}$ is a martingale we obtain that almost surely
    \begin{align*}
    &\lim_{n\rightarrow\infty}\frac{1}{n}\sum_{i=1}^n\left(\mathbb{E}[(Y_i-Y_{i-1})^2\mathbbm{1}_{\{|Y_i-Y_{i-1}|\leq Z_i\}}\;|\;\mathcal{F}_{i-1}]-\mathbb{E}[(Y_i-Y_{i-1})\mathbbm{1}_{\{|Y_i-Y_{i-1}|\leq Z_i\}}\;|\;\mathcal{F}_{i-1}]^2\right)\\&=\lim_{n\rightarrow\infty}\frac{V_n}{n}=4p(1-p),
    \end{align*}
    where the limit follows from Lemma \ref{lem:quadratic-variation}.

    \item Using again that the increments of our martingale are bounded we obtain almost surely
    $$\sum_{i=1}^n\frac{1}{i^2}\mathbb{E}[(Y_i-Y_{i-1})^4\mathbbm{1}_{\{|Y_i-Y_{i-1}|\leq Z_i\}}\;|\;\mathcal{F}_{i-1}]\leq \sum_{i=1}^n\frac{C^4}{i^2}<\infty.$$

    \item By our choice of the interpolation times it follows
    $$\lim_{n\rightarrow\infty}\frac{\sqrt{n+1}}{\sqrt{n}}=1,\hspace{1cm}\lim_{n\rightarrow\infty}\sqrt{n}=\infty.$$
\end{enumerate}

We can thus apply the functional LIL for martingales from Lemma \ref{lem:martingale-LIL}. We collect here the result.

\begin{lemma}\label{lem:martingale-term-convergence}
    For the sequence of paths  $(\mathcal{Y}_n)_{n\in\N}$ and $(\mathcal{W}_n)_{n\in\N}$ it holds almost surely that they are relatively compact in $C([0,1])$ and the set of their accumulation points coincides exactly with $2\sqrt{p(1-p)}\mathcal{K}.$
\end{lemma}
\begin{proof}
    We have already checked that the conditions of Lemma \ref{lem:martingale-LIL} apply to the sequence of paths $(\mathcal{Y}_n)_{n\in\N}$. The analogue statement for the sequence $(\mathcal{W}_n)_{n\in\N}$ follows from the fact that for all $n\in\N$ the random variable $W_n$ is given by adding a bounded increment to $Y_n$. This increment vanishes in the limit due to the renormalization by dividing by $\sqrt{2n\log\log(n)}$.
\end{proof}
\subsection{Application of the Perturbation Mapping}
Recall the definition of the numbers $\alpha,\beta<1$ from Lemma \ref{lem:martingale-decomp} as well as the perturbation mapping $\Phi_{\alpha,\beta}$ from (\ref{eq:perturbation-mapping}). We now have that
$X_n=W_n+\alpha\overline{X}_n+\beta\underline{X}_n.$
If we translate this to our interpolated paths, it holds for all $n\in\N$ that $\mathcal{X}_n=\Phi_{\alpha,\beta}(\mathcal{W}_n).$ From this identification, we can now easily follow the functional LIL for the $p$-walk paths.

\begin{lemma}\label{lem:functional-LIL-p-walk}
    For the sequence of paths $(\mathcal{X}_n)_{n\in\N}$ it holds almost surely that they are relatively compact in $C([0,1])$ and the set of its accumulation points coincides exactly with $\Phi_{\alpha,\beta}(2\sqrt{p(1-p)}\mathcal{K}).$
\end{lemma}
\begin{proof}
    Since $\mathcal{X}_n=\Phi_{\alpha,\beta}(\mathcal{W}_n)$ and $\Phi_{\alpha,\beta}$ preservers relative compactness (see \cite[Lemma 2.1.]{D96}), it follows from Lemma \ref{lem:martingale-term-convergence} that the sequence of paths $(\mathcal{X}_n)_{n\in\N}$ is relatively compact. Thus, by the Continuous Mapping Theorem the set of accumulation points is given by $\Phi_{\alpha,\beta}(2\sqrt{p(1-p)}\mathcal{K}).$
\end{proof}

\section{Constants in the Law of the Iterated Logarithm}\label{sec:var}
We have so far shown that the linearly interpolated paths of the $p$-walk after proper rescaling are relatively compact in $C([0,1])$, and we were also able to determine its set of accumulation points. It now remains to use this result in order to determine the exact constants in the LIL for the $p$-walk. For this end, let us notice that it holds
\begin{align}\label{eq:supremum}
\limsup_{n\rightarrow\infty}\frac{X_n}{\sqrt{2n\log\log(n)}}=\limsup_{n\rightarrow\infty}\mathcal{X}_n(1)=\sup_{g\in\Phi_{\alpha,\beta}(2\sqrt{p(1-p)}\mathcal{K})}g(1),
\end{align}
where we used the set of accumulation points as derived in Lemma \ref{lem:functional-LIL-p-walk}. Thus, determining the exact constant of the classical LIL reduces to solving the calculus of variations problem of maximizing the endpoint of all functions in the set $\Phi_{\alpha,\beta}(2\sqrt{p(1-p)}\mathcal{K})$. Notice, that similarly we can determine the constant for the $\liminf$.

\begin{lemma}\label{lem:variation-calculus}
    It holds that
    $$\sup_{g\in\Phi_{\alpha,\beta}(\mathcal{K})}g(1)=\frac{1}{1-\alpha},\hspace{1cm}\inf_{g\in\Phi_{\alpha,\beta}(\mathcal{K})}g(1)=-\frac{1}{1-\beta}.$$
\end{lemma}
\begin{proof}
    We will only show how to determine the value of the supremum. The infimum works analogously. For ease of notation, let us write  $\mathcal{K}'=\Phi_{\alpha,\beta}(\mathcal{K}).$

    Let now $g\in\mathcal{K}'$, then consider the running supremum $\overline{g}$. Notice that it trivially holds that $\overline{g}(1)\geq g(1)$, thus our first goal is to show $\overline{g}\in\mathcal{K}'.$ We thus need to find a funciton $\tilde{f}\in\mathcal{K}$ such that $\overline{g}=\Phi_{\alpha,\beta}(\tilde{f})$. We define $\tilde{f}$ as the exact inverse mapping of $\overline{g}$, that is
    $\tilde{f}=(1-\alpha)\overline{g}.$
    In the proof of Lemma \ref{lem:absolute-cont-preserved} we have already seen that $\overline{g}$ must also be absolutely continuous, thus we obtain the absolute continuity of $\tilde{f}$. Since the derivative of $\tilde{f}$ agrees with the derivative of $f$ on the set $\{g=\overline{g}\}$ and is $0$ everywhere else, this implies $\tilde{f}\in\mathcal{K}.$ Since $\tilde{f}$ was defined as the exact inverse of $\overline{g}$, this implies $\overline{g}\in\mathcal{K}'$. Thus we can center our view only on the monotone functions in $\mathcal{K}'$.

    Let now $g\in\mathcal{K}'$ be a monotone increasing function. Thus it holds $\overline{g}=g$ and $\underline{g}=0$. Let us choose the pre-image $f\in\mathcal{K}$, then we have
    $f=(1-\alpha)g.$
    Thus, $g$ fulfills the following energy constraint
    $$\int_0^1(g'(t))^2dt\leq \frac{1}{(1-\alpha)^2}\int_0^1(f'(t))^2dt\leq \frac{1}{(1-\alpha)^2}.$$
    From Cauchy-Schwartz inequality and the absolute continuity of $g$, it thus follows
    $$g(1)=\int_0^1g'(t)dt\leq\left(\int_0^1(g'(t))^2dt\right)^{1/2}\leq\frac{1}{1-\alpha}.$$
    Notice that this upper bound can be achieved by the function $g(t)=\frac{1}{1-\alpha}t$, and thus the claim follows.
\end{proof}
With this result, we can now prove our main theorem.
\begin{proof}[Proof of Theorem \ref{thm:main}]
    As we have seen in Equation (\ref{eq:supremum}), it suffices to find
    $$\sup_{g\in\Phi_{\alpha,\beta}(2\sqrt{p(1-p)}\mathcal{K})}g(1).$$
    Since the set is only scaled by the constant factor $2\sqrt{p(1-p)}$, we immediately obtain our result from Lemma \ref{lem:variation-calculus},
    $$\limsup_{n\rightarrow\infty}\frac{X_n}{\sqrt{2n\log\log(n)}}=\frac{2\sqrt{p(1-p)}}{1-\alpha}.$$
    The limes inferior follows analogously.
\end{proof}
\section{Conclusion}
In this paper, we have established the exact almost sure asymptotic boundaries for the $p$-rotor walk on $\mathbb{Z}$. By combining the discrete martingale representation of \cite{HLSH18} with Strassen’s functional limit theory and the calculus of variations, we translated the probabilistic fluctuations of the walk into a deterministic problem over the perturbed Strassen set. This approach elegantly completely characterizes the pathwise extremes of the walk without requiring path-dependent local time bounds. This result opens several natural avenues for future investigations.

\textbf{Exact Convergence Rates.} Our result from Theorem \ref{thm:main} adds to the central limit behaviour of $p$-walks from \cite{HLSH18}. However, while these results establish convergence and identify the correct limit, none of them give insight into how fast the process converges to its respective limits. A natural direction is thus, to analyze the speed of convergence in the central limit theorem of $p$-walks, and to prove Berry-Esseen type bounds for the central limit behaviour.

\textbf{Multi-Dimensional Analogues and Recurrence.} The simple geometry of the one-dimensional line allows us to analyze the $p$-walk in great detail. However, if we move to higher dimensions, specifically $\Z^2$, the analysis of $p$-walks becomes considerably harder. In fact, besides the value $p=1/2$, it is not clear whether $p$-walks in $\Z^2$ are recurrent or not, although they are strongly believed to be recurrent. Furthermore, considering scaling limits and the central limit behaviour, there exists no natural candidate for a scaling limit that could take the place of the perturbed Brownian motion in higher dimensions.

\textbf{Analysis of Locally Markov walks.} The $p$-walks considered in this paper fall into the wider category of locally Markov walks. A locally Markov walk $(X_n)_{n\in\N}$ is a stochastic process, whose next step solely depends on the last action performed at its current location. That is, the local routings of the particle form Markov chains at every vertex of the underlying graph.

Such locally Markov walks have been analyzed in \cite{CGLP21} and \cite{KLSH26}. It would be highly interesting, to study this class of stochastic process in general, and to establish tools to answer questions about locally Markov walks, such as proving recurrence or transience on infinite graphs, or to  establish precise mixing time bounds on finite graphs.

\bibliographystyle{alpha}
\bibliography{lit}

\textsc{Robin Kaiser}, Departement of Mathematics, CIT, Technische Universität München, Boltzmannstr. 3, D-85748 Garching bei München, Germany. \texttt{ro.kaiser@tum.de}
\end{document}